\documentclass[12pt]{amsart}
\usepackage{amsmath,amsthm,amsfonts,amssymb,mathrsfs}
\date{\today}

\usepackage[cp1251]{inputenc}         

\usepackage{amsmath,amsthm,amsfonts,amssymb,mathrsfs}
\usepackage{eucal}

\usepackage{color}

\usepackage{amssymb,cite}
\usepackage[colorlinks,plainpages,citecolor=magenta, linkcolor=blue, backref]{hyperref}

\usepackage{hyperref}

\newtheorem{theorem}{Theorem}[section]
\newtheorem{lemma}[theorem]{Lemma}

\newtheorem{proposition}[theorem]{Proposition}
\newtheorem{corollary}[theorem]{Corollary}

\theoremstyle{definition}

\newtheorem{example}[theorem]{Example}

\theoremstyle{remark}
\newtheorem{remark}[theorem]{Remark}

\numberwithin{equation}{section}

\begin{document}

\title[On locally compact shift continuous topologies on the semigroup $\boldsymbol{B}_{[0,\infty)}$ ...]{On locally compact shift continuous topologies on the semigroup $\boldsymbol{B}_{[0,\infty)}$ with an adjoined  compact ideal}

\author[O.~Gutik and M.~Khylynskyi]{Oleg~Gutik and Markian~Khylynskyi}
\address{Ivan Franko National University of Lviv, Universytetska 1, Lviv, 79000, Ukraine}
\email{oleg.gutik@lnu.edu.ua, ogutik@gmail.com, markian.khylynskyi@lnu.edu.ua}


\begin{abstract}
Let $\boldsymbol{B}_{[0,\infty)}$ be the semigroup which is defined in the Ahre paper \cite{Ahre=1981}. The semigroup $\boldsymbol{B}_{[0,\infty)}$ with the induced usual topology $\tau_u$ from $\mathbb{R}^2$, with the topology $\tau_L$ which is generated by the natural partial order on $\boldsymbol{B}_{[0,\infty)}$, and the discrete topology  are denoted by $\boldsymbol{B}^1_{[0,\infty)}$, $\boldsymbol{B}^2_{[0,\infty)}$, and $\boldsymbol{B}^{\mathfrak{d}}_{[0,\infty)}$, respectively. We show that if $S_1^I$ ($S_2^I$) is a Hausdorff locally compact semitopological  semigroup $\boldsymbol{B}^1_{[0,\infty)}$ ($\boldsymbol{B}^2_{[0,\infty)}$) with an adjoined compact ideal $I$ then either $I$ is an open subset of $S_1^I$ ($S_2^I$) or the semigroup $S_1^I$ ($S_2^I$) is compact. Also, we proved that if $S_{\mathfrak{d}}^I$ is a Hausdorff locally compact semitopological  semigroup $\boldsymbol{B}^{\mathfrak{d}}_{[0,\infty)}$  with an adjoined compact ideal $I$ then $I$ is an open subset of $S_{\mathfrak{d}}^I$.
\end{abstract}

\maketitle


\section{Introduction and preliminaries}

In this paper we shall follow the terminology of \cite{Carruth-Hildebrant-Koch=1983, Carruth-Hildebrant-Koch=1986, Clifford-Preston=1961, Clifford-Preston=1967, Engelking=1989, Lawson=1998, Ruppert=1984}.

A semigroup $S$ is called {\it inverse} if for any
element $x\in S$ there exists a unique $x^{-1}\in S$ such that
$xx^{-1}x=x$ and $x^{-1}xx^{-1}=x^{-1}$. The element $x^{-1}$ is
called the {\it inverse of} $x\in S$. If $S$ is an inverse
semigroup, then the function $\operatorname{inv}\colon S\to S$
which assigns to every element $x$ of $S$ its inverse element
$x^{-1}$ is called the {\it inversion}.
On an inverse semigroup $S$ the semigroup operation  determines the following partial order $\preccurlyeq$: $s\preccurlyeq t$ if and only if there exists $e\in E(S)$ such that $s=te$. This partial order is called the natural partial order on $S$.

\begin{remark}\label{remark-1.1}
For arbitrary elements $s,t$ of an inverse semigroup $S$ the following conditions are equivalent:
\begin{equation*}
  (\alpha)~s\preccurlyeq t; \quad (\beta)~s=ss^{-1}t; \quad (\gamma)~s= t s^{-1}s,
\end{equation*}
(see \cite[Chap. 3]{Lawson=1998}).
\end{remark}

A topological space $X$ is called  \emph{locally compact} if every poin $x$ of $X$ has an open neighbourhood with the compact closure.

A (\emph{semi})\emph{topological} \emph{semigroup} is a topological space with a (separately) continuous semigroup operation. An inverse topological semigroup with continuous inversion is called a \emph{topological inverse semigroup}.

A topology $\tau$ on a semigroup $S$ is called:
\begin{itemize}
  \item a \emph{semigroup} topology if $(S,\tau)$ is a topological semigroup;
  \item an \emph{inverse semigroup} topology if $(S,\tau)$ is a topological inverse semigroup;
  \item a \emph{shift-continuous} topology if $(S,\tau)$ is a semitopological semigroup.
\end{itemize}

The bicyclic monoid ${\mathscr{C}}(p,q)$ is the semigroup with the identity $1$ generated by two elements $p$ and $q$ subjected only to the condition $pq=1$. The semigroup operation on ${\mathscr{C}}(p,q)$ is determined as
follows:
\begin{equation*}
    q^kp^l\cdot q^mp^n=q^{k+m-\min\{l,m\}}p^{l+n-\min\{l,m\}}.
\end{equation*}
It is well known that the bicyclic monoid ${\mathscr{C}}(p,q)$ is a bisimple (and hence simple) combinatorial $E$-unitary inverse semigroup and every non-trivial congruence on ${\mathscr{C}}(p,q)$ is a group congruence \cite{Clifford-Preston=1961}.

The bicyclic monoid admits only the discrete semigroup Hausdorff topology \cite{Eberhart-Selden=1969}. Bertman and  West in \cite{Bertman-West=1976} extended this result for the case of Hausdorff semitopological semigroups. If a Hausdorff (semi)topological semigroup $T$ contains the bicyclic monoid ${\mathscr{C}}(p,q)$ as a dense proper semigroup then $T\setminus {\mathscr{C}}(p,q)$ is a closed ideal of $T$ \cite{Eberhart-Selden=1969, Gutik=2015}. Moreover, the closure of ${\mathscr{C}}(p,q)$ in a locally compact topological inverse semigroup can be obtained (up to isomorphism) from ${\mathscr{C}}(p,q)$ by adjoining the additive group of integers in a suitable way \cite{Eberhart-Selden=1969}.

Stable and $\Gamma$-compact topological semigroups do not contain the bicyclic monoid~\cite{Anderson-Hunter-Koch=1965, Hildebrant-Koch=1986, Koch-Wallace=1957}. The problem of embedding the bicyclic monoid into compact-like topological semigroups was studied in \cite{Banakh-Dimitrova-Gutik=2009, Banakh-Dimitrova-Gutik=2010, Bardyla-Ravsky=2020, Gutik-Repovs=2007}.

In \cite{Ahre=1981}  Ahre considered the following semigroup. Let $[0,\infty)$ be the set of all non-negative real numbers. The set $\boldsymbol{B}_{[0,\infty)}=[0,\infty)\times [0,\infty)$ with the following binary operation
\begin{equation*}
  (a,b)(c,d)=(a+c-\min\{b,c\},b+d-\min\{b,c\})=
  \left\{
    \begin{array}{ll}
      (a+c-b,d), & \hbox{if~} b<c; \\
      (a,d),     & \hbox{if~} b=c; \\
      (a,b+d-c)  & \hbox{if~} b>c.
    \end{array}
  \right.
\end{equation*}
Then $\boldsymbol{B}_{[0,\infty)}$ is a bisimple inverse semigroup. The semigroup $\boldsymbol{B}_{[0,\infty)}$ and the bicyclic monoid ${\mathscr{C}}(p,q)$ are partial cases of bicyclic extensions of linearly ordered groups which are presented in \cite{Fortunatov=1976, Fotedar=1974, Fotedar=1978, Gutik-Pagon-Pavlyk=2011}. By $\boldsymbol{B}^1_{[0,\infty)}$ we denote the semigroup $\boldsymbol{B}_{[0,\infty)}$ with the usual topology. It is obvious that $\boldsymbol{B}^1_{[0,\infty)}$ is a locally compact topological inverse semigroup \cite{Ahre=1981}. In \cite{Ahre=1983, Ahre=1986} it is shown that the closure of $\boldsymbol{B}^1_{[0,\infty)}$ in a locally compact topological inverse semigroup can be obtained (up to isomorphism) from $\boldsymbol{B}^1_{[0,\infty)}$ by adjoining the additive group of reals in a suitable way.

For any non-negative real number $\alpha$ we denote the following subsets in $\boldsymbol{B}_{[0,\infty)}$:
\begin{equation*}
  L_\alpha^+=\{(x,x+\alpha)\colon x\geqslant 0\} \qquad \hbox{and} \qquad L_\alpha^-=\{(x+\alpha,x)\colon x\geqslant 0\}.
\end{equation*}
It obvious that $\boldsymbol{B}_{[0,\infty)}=\bigsqcup_{\alpha\geqslant 0}L_{\alpha}^+\sqcup \bigsqcup_{\alpha>0}L_\alpha^-$ and $L_0^+=L_0^-$. Put $\tau_L$ be a topology on $\boldsymbol{B}_{[0,\infty)}$ which is generating by the bases
\begin{equation*}
  \mathcal{B}(x,x+\alpha)=\left\{U_\varepsilon(x,x+\alpha)=\left\{(x+y,x+y+\alpha)\in L_\alpha^+\colon |y|<\varepsilon\right\}\colon \varepsilon>0\right\}
\end{equation*}
and
\begin{equation*}
  \mathcal{B}(x+\alpha,x)=\left\{U_\varepsilon(x+\alpha,x)=\left\{(x+y+\alpha,x+y)\in L_\alpha^-\colon |y|<\varepsilon\right\}\colon \varepsilon>0\right\}
\end{equation*}
at any points $(x,x+\alpha)\in L_\alpha^+$ and $(x+\alpha,x)\in L_\alpha^-$, respectively, for arbitrary $\alpha\in[0,+\infty)$. The semigroup $\boldsymbol{B}_{[0,\infty)}$ with the topology $\tau_L$ is denoted by $\boldsymbol{B}^2_{[0,\infty)}$. 
The definitions of the topology $\tau_L$ and the natural partial order on $\boldsymbol{B}_{[0,\infty)}$ imply that $\tau_L$ is generated by the natural partial order of $\boldsymbol{B}_{[0,\infty)}$ (see \cite{Gierz-Hofmann-Keimel-Lawson-Mislove-Scott=2003}).
We observe that $\boldsymbol{B}^2_{[0,\infty)}$ is a Hausdorff locally compact topological inverse semigroup \cite{Ahre=1989}. Moreover for any non-negative real number $\alpha$, $L_\alpha^+$ and $L_\alpha^-$ are open-and-closed subsets of $\boldsymbol{B}^2_{[0,\infty)}$ which are homeomorphic to $[0,+\infty)$ with the usual topology, i.e., $\boldsymbol{B}^2_{[0,\infty)}=\displaystyle\bigoplus_{\alpha\geqslant 0}L_{\alpha}^+\oplus \bigoplus_{\alpha>0}L_\alpha^-$. The closure of the topological inverse semigroup $\boldsymbol{B}^2_{[0,\infty)}$ in (locally compact) topological semigroups is studied in \cite{Ahre=1989}.

By $\boldsymbol{B}^\mathfrak{d}_{[0,\infty)}$ we denote the semigroup $\boldsymbol{B}_{[0,\infty)}$ with the discrete topology. It is obvious that $\boldsymbol{B}^\mathfrak{d}_{[0,\infty)}$ is a locally compact topological inverse semigroup.

In the paper \cite{Gutik=2015} it is proved that every Hausdorff locally compact shift-continuous topology on the bicyclic monoid with adjoined zero is either compact or discrete. This result was extended by Bardyla onto the a polycyclic monoid \cite{Bardyla=2016} and graph inverse semigroups \cite{Bardyla=2018}, and by Mokrytskyi onto the monoid of order isomorphisms between principal filters of $\mathbb{N}^n$ with adjoined zero \cite{Mokrytskyi=2019}.
In \cite{Gutik-Khylynskyi=2022} the results of paper \cite{Gutik=2015} onto the monoid $\mathbf{I}\mathbb{N}_{\infty}$ of all partial cofinite isometries of positive integers with adjoined zero are extended. In \cite{Gutik-Mykhalenych=2023} the similar dichotomy was proved for so called bicyclic extensions $\boldsymbol{B}_{\omega}^{\mathscr{F}}$  when a family $\mathscr{F}$ consists of inductive non-empty subsets of~$\omega$.
Algebraic properties on a group $G$ such that if the discrete group $G$ has these properties then every locally compact shift continuous topology on
$G$ with adjoined zero is either compact or discrete studied in \cite{Maksymyk=2019}. Also, in \cite{Gutik-Maksymyk=2019} it is proved that the extended bicyclic semigroup $\mathscr{C}_\mathscr{\mathbb{Z}}^0$ with adjoined zero admits distinct $\mathfrak{c}$-many  shift-continuous topologies, however every Hausdorff locally compact semigroup topology on $\mathscr{C}_\mathscr{\mathbb{Z}}^0$ is discrete. In \cite{Bardyla=2023} Bardyla proved that a Hausdorff locally compact semitopological semigroup McAlister Semigroup $\mathcal{M}_1$ is either compact or discrete. However, this dichotomy does not hold for the McAlister Semigroup $\mathcal{M}_2$ and moreover, $\mathcal{M}_2$ admits continuum many different Hausdorff locally compact inverse semigroup topologies \cite{Bardyla=2023}.

In this paper we extend the results of paper \cite{Gutik=2015} onto the topological monoids $\boldsymbol{B}^1_{[0,\infty)}$ and $\boldsymbol{B}^2_{[0,\infty)}$. In particular we show that if $S_1^I$ ($S_2^I$) is a Hausdorff locally compact semitopological  semigroup $\boldsymbol{B}^1_{[0,\infty)}$ ($\boldsymbol{B}^2_{[0,\infty)}$) with an adjoined compact ideal $I$ then either $I$ is an open subset of $S_1^I$ ($S_2^I$) or the semigroup $S_1^I$ ($S_2^I$) is compact. Also, we proved that if $S_{\mathfrak{d}}^I$ is a Hausdorff locally compact semitopological  semigroup $\boldsymbol{B}^{\mathfrak{d}}_{[0,\infty)}$  with an adjoined compact ideal $I$ then $I$ is an open subset of $S_{\mathfrak{d}}^I$.

\section{A locally compact semigroup $\boldsymbol{B}^1_{[0,\infty)}$ with an adjoined compact ideal}

Later in this section  by $S_1^I$ we denote a Hausdorff locally compact semitopological  semigroup which is the semigroup $\boldsymbol{B}^1_{[0,\infty)}$ with an adjoined non-open  compact ideal $I$.

\begin{lemma}\label{lemma-2.1}
Let $S$ be a  Hausdorff locally compact semitopological semigroup with a compact ideal $I$.
Then for any open neighbourhood $U(I)$ of the ideal $I$ and any  $x\in S$ there exists an open neighbourhood $V(I)$ of $I$ with the compact closure $\overline{V(I)}$ such that $x\cdot V(I)\subseteq U(I)$ and $V(I)\cdot x\subseteq U(I)$.
\end{lemma}

\begin{proof}
Fix an arbitrary open neighbourhood $U(I)$ of the ideal $I$ and any  $x\in S$. Since $I$ is an ideal of $S$, for any $\alpha\in I$ there exists $\beta\in I$ such that $x\cdot \alpha=\beta$. Since $U(I)$ is an open neighbourhood of $\beta$, separate continuity of the semigroup operation in $S$ implies that there exists an open neighbourhood $V(\alpha)$ of $\alpha$ in $S$ such that $x\cdot V(\alpha)\subseteq U(I)$. The local compactness of the space $S$ implies that without loss of generality we may assume that the neighbourhood $V(\alpha)$ has the compact closure $\overline{V(\alpha)}$. Then the family $\{V(\alpha)\colon \alpha\in I\}$ is an open cover of $I$. Since $I$ is compact,  $I\subseteq V(\alpha_1)\cup\ldots\cup V(\alpha_n)$ for finitely many $\alpha_1,\ldots,\alpha_n\in I$. Put $V_1(I)=V(\alpha_1)\cup\ldots\cup V(\alpha_n)$. Then $\overline{V_1(I)}=\overline{V(\alpha_1)}\cup\ldots\cup \overline{V(\alpha_n)}$ is a compact subset of $S$ such that $x\cdot V_1(I)\subseteq U(I)$. Similarly we get that there exists an open neighbourhood $V_2(I)$ of $I$ with the compact closure $\overline{V_2(I)}$ such that $V_2(I)\cdot x\subseteq U(I)$. Put $V(I)=V_1(I)\cap V_2(I)$. Then $V(I)$ is an open neighbourhood of $I$ with the compact closure $\overline{V(I)}=\overline{V_1(I)}\cap \overline{V_2(I)}$ such that $x\cdot V(I)\subseteq U(I)$ and $V(I)\cdot x\subseteq U(I)$.
\end{proof}

A subset $A$ of $\boldsymbol{B}_{[0,\infty)}$ is called \emph{unbounded} if for any positive real number $a$ there exist $(x,y)\in A$ such that  $x\geqslant a$ and $y\geqslant a$.

\begin{lemma}\label{lemma-2.2}
For any open neighbourhood $U(I)$ of the ideal $I$ in $S_1^I$ the set $U(I)\cap \boldsymbol{B}_{[0,\infty)}$ is unbounded.
\end{lemma}

\begin{proof}
Suppose to the contrary that there exists a positive real number $m$ such that $x<m$ or $y<m$ for any $(x,y)\in U(I)\cap \boldsymbol{B}_{[0,\infty)}$. Lemma~\ref{lemma-2.1} implies that there exists an open neighbourhood $V(I)\subseteq U(0)$ of $I$ such that
\begin{equation*}
  V(I)\cdot (0,2m)\subseteq U(I) \qquad \hbox{and} \qquad  (2m,0)\cdot V(I)\subseteq U(I).
\end{equation*}
Since the ideal $I$ is a non-open subset of $S_1^I$, we have that
\begin{equation*}
  (V(I)\cdot (0,2m))\cap U(I)\neq I\neq ((2m,0)\cdot V(I))\cap U(I).
\end{equation*}
Then the definition of the semigroup operation on $\boldsymbol{B}_{[0,\infty)}$ implies that there exists $(a,b)\in U(I)$ such that $a>m$ and $b>m$, which implies the statement of the lemma.
\end{proof}

\begin{proposition}\label{proposition-2.3}
For any open neighbourhood $U(I)$ of the ideal $I$ in $S_1^I$ there exists a compact subset $A_a=[0,a]\times[0,a]$ in $\boldsymbol{B}^1_{[0,\infty)}$ such that $S_1^I\setminus U(I)\subseteq A_a$.
\end{proposition}

\begin{proof}
Suppose to the contrary that there exists an open  neighbourhood $U(I)$ of the ideal $I$ in $S_1^I$ such that $S_1^I\setminus U(I)\nsubseteq A_n$ for any positive integer $n$. By Lemma~\ref{lemma-2.1} without loss of generality we may assume that the closure $\overline{U(I)}$ is a compact subset of $S_1^I$. By Lemma~\ref{lemma-2.2} the set  $U(I)\cap \boldsymbol{B}_{[0,\infty)}$ is unbounded in $\boldsymbol{B}_{[0,\infty)}$. Since $\boldsymbol{B}^1_{[0,\infty)}\cap U(I)$ is an open subset in $\boldsymbol{B}^1_{[0,\infty)}$, the assumption of the proposition implies that for any positive integer $n$ there exists an element $(x_n,y_n)\in \overline{U(I)}\setminus U(I)$ such that $(x_n,y_n)\notin A_n$. This implies that the set $\overline{U(I)}\setminus U(I)$ is unbounded in $\boldsymbol{B}_{[0,\infty)}$. But $\overline{U(I)}\setminus U(I)$ is a compact subspace of the metric space $\boldsymbol{B}^1_{[0,\infty)}$, a contradiction.
\end{proof}

Proposition \ref{proposition-2.3}  implies the following theorem.

\begin{theorem}\label{theorem-2.4}
Let $S_1^I$ be a Hausdorff locally compact semitopological  semigroup $\boldsymbol{B}^1_{[0,\infty)}$ with an adjoined compact ideal $I$. Then either $I$ is an open subset of $S_1^I$ or the semigroup $S_1^I$ is compact.
\end{theorem}

Example \ref{example-2.5} and Proposition \ref{proposition-2.6} show that if the ideal $I$ of the semigroup $S_1^I$ is trivial, i.e., the ideal $I$ is a singleton, then the semigroup  $S_1^I$ admits the unique Hausdorff compact shift-continuous topology.

\begin{example}\label{example-2.5}
Let $S^{\boldsymbol{0}}_1$ be the semigroup $\boldsymbol{B}^1_{[0,\infty)}$ with an adjoined zero $\boldsymbol{0}$.
We extend the topology of  $\boldsymbol{B}^1_{[0,\infty)}$  up to a compact topology $\tau^1_{\textsf{Ac}}$ on $S^{\boldsymbol{0}}_1$ in the following way. We define
\begin{equation*}
  \mathscr{B}^1_{\textsf{Ac}}(\boldsymbol{0})=\left\{U_n(\boldsymbol{0})=\left\{0\right\}\cup\left\{(x,y)\colon x>n \hbox{~or~} y>n\right\}\colon n\in\mathbb{N}\right\}
\end{equation*}
is the system of open neighbourhoods of zero in $\tau^1_{\textsf{Ac}}$.
\end{example}

\begin{proposition}\label{proposition-2.6}
$(S^{\boldsymbol{0}}_1,\tau^1_{\textsf{Ac}})$ is a compact Hausdorff semitopological semigroup with continuous inversion.
\end{proposition}

\begin{proof}
By \cite{Ahre=1983, Ahre=1986}, $\boldsymbol{B}^1_{[0,\infty)}$ is a topological inverse semigroup, and hence  it sufficient to show that the  semigroup operation on  $(S^{\boldsymbol{0}}_1,\tau^1_{\textsf{Ac}})$ is separately continuous at zero.

It is obvious that $\boldsymbol{0}\cdot U_n(\boldsymbol{0})=U_n(\boldsymbol{0})\cdot \boldsymbol{0}=\{\boldsymbol{0}\}\subseteq U_n(\boldsymbol{0})$ for any positive integer $n$.

Next we shall show that $(x,y)\cdot U_{2n}(\boldsymbol{0})\subseteq U_n(\boldsymbol{0})$ for any positive integer $n>\max\{x, y\}+1$. We consider the possible cases.

1. Suppose that $a>2n$. Then for any  $b\in\mathbb{R}$ the equality
\begin{equation}\label{eq-2.1}
  (x,y)(a,b)=
  \left\{
    \begin{array}{ll}
      (x-y+a,b), & \hbox{if~} y<a; \\
      (x,b),     & \hbox{if~} y=a; \\
      (x,y-a+b), & \hbox{if~} y>a,
    \end{array}
  \right.
\end{equation}
implies that
$
  (x,y)(a,b)=(x-y+a,b).
$
By the assumptions $n>\max\{x, y\}+1$ and $a>2n$, we get that
$
x-y+a>-n+2n=n,
$
and hence $(x-y+a,b)\in U_n(\boldsymbol{0})$.

2. Suppose that $n\leqslant a\leqslant 2n$ and $b>2n$. By \eqref{eq-2.1} we have that
$
  (x,y)(a,b)=(x-y+a,b).
$
The assumption $n>\max\{x, y\}+1$ implies that
$
x-y+a>-n+n=0.
$
Since $b>2n$ we get that $(x-y+a,b)\in U_n(\boldsymbol{0})$.

3. Suppose that $0\leqslant a< n$ and $b>2n$. By \eqref{eq-2.1} we have that
\begin{equation*}
(x,y)(a,b)=(x-y+a,b)\in U_n(\boldsymbol{0})
\end{equation*}
in the case when $y< a$, and if $y\geqslant a$ then
$
y-a+b>2n,
$
and hence $(x,y-a+b)\in U_n(\boldsymbol{0})$.

 Similar arguments and the equality
\begin{equation*}
  (a,b)(x,y)=
  \left\{
    \begin{array}{ll}
      (a-b+x,y), & \hbox{if~} b<x; \\
      (a,y),     & \hbox{if~} b=x; \\
      (a,b-x+y)  & \hbox{if~} b>x,
    \end{array}
  \right.
\end{equation*}
imply that for any positive integer $n>\max\{x, y\}+1$ the inclusion $U_{2n}(\boldsymbol{0})\cdot(x,y) \subseteq U_n(\boldsymbol{0})$ holds. The above inclusions imply that the semigroup operation on $(S^{\boldsymbol{0}}_1,\tau^1_{\textsf{Ac}})$ is separate continuous.

Since $\left(U_{n}(\boldsymbol{0})\right)^{-1}=U_{n}(\boldsymbol{0})$ for any $n\in\mathbb{N}$ the inversion on $(S^{\boldsymbol{0}}_1,\tau^1_{\textsf{Ac}})$ is continuous.

It is obvious that $\tau^1_{\textsf{Ac}}$ is a compact Hausdorff topology on $S^{\boldsymbol{0}}_1$. Moreover $(S^{\boldsymbol{0}}_1,\tau^1_{\textsf{Ac}})$ is  the one-point Alexandroff compactification of the locally compact space $\boldsymbol{B}^1_{[0,\infty)}$ such that the singleton set $\{\boldsymbol{0}\}$ which consists of the zero of $S^{\boldsymbol{0}}_1$ is its remainder.
\end{proof}

Theorem \ref{theorem-2.4} and Proposition \ref{proposition-2.6} imply the following theorem.

\begin{theorem}\label{theorem-2.7}
Let $S^{\boldsymbol{0}}_1$ be a Hausdorff locally compact semitopological  semigroup $\boldsymbol{B}^1_{[0,\infty)}$ with an adjoined zero $\boldsymbol{0}$. Then either $\boldsymbol{0}$ is an isolated point of $S^{\boldsymbol{0}}_1$ or the topology of $S^{\boldsymbol{0}}_1$ coincides with $\tau^1_{\textsf{Ac}}$.
\end{theorem}

Since the bicyclic monoid does not embeds into any Hausdorff compact topological semigroup \cite{Anderson-Hunter-Koch=1965} and the semigroup contains many isomorphic copies of the bicyclic semigroup, Theorems \ref{theorem-2.4} and \ref{theorem-2.7} imply the following corollaries.

\begin{corollary}\label{corollary-2.8}
Let $S_1^I$ be a Hausdorff locally compact topological  semigroup $\boldsymbol{B}^1_{[0,\infty)}$ with an adjoined compact ideal $I$. Then $I$ is an open subset of $S_1^I$.
\end{corollary}

\begin{corollary}\label{corollary-2.9}
Let $S^{\boldsymbol{0}}_1$ be a Hausdorff locally compact topological  semigroup $\boldsymbol{B}^1_{[0,\infty)}$ with an adjoined zero $\boldsymbol{0}$. Then $\boldsymbol{0}$ is an isolated point of $S^{\boldsymbol{0}}_1$.
\end{corollary}

\section{A locally compact semigroup $\boldsymbol{B}^2_{[0,\infty)}$ with an adjoined compact ideal}

Later in this section by $S_2^I$ we denote a Hausdorff locally compact semitopological  semigroup which is the semigroup $\boldsymbol{B}^2_{[0,\infty)}$ with an adjoined non-open  compact ideal $I$.

The proof of Lemma~\ref{lemma-3.1} is similar to Lemma~\ref{lemma-2.2}.

\begin{lemma}\label{lemma-3.1}
For any open neighbourhood $U(I)$ of the ideal $I$ in $S_1^I$ the set $U(I)\cap \boldsymbol{B}_{[0,\infty)}$ is unbounded.
\end{lemma}

\begin{lemma}\label{lemma-3.2}
Let $U(I)$ be any open neighbourhood of the ideal $I$ in $S_2^I$ with the compact closure $\overline{U(I)}$. Then there exist finite subsets $B$ and $C$ of non-negative real numbers such that
\begin{equation*}
S_2^I\setminus U(I)\subseteq \bigsqcup_{\alpha\in B}L_{\alpha}^+\sqcup \bigsqcup_{\alpha\in C}L_\alpha^-.
\end{equation*}
\end{lemma}

\begin{proof}
Since $\overline{U(I)}\setminus U(I)$ is compact subset in $S_2^I$, $\overline{U(I)}\setminus U(I)$ is compact subset in $\boldsymbol{B}^2_{[0,\infty)}$.  The equality $\boldsymbol{B}^2_{[0,\infty)}=\bigoplus_{\alpha\in[0,+\infty)}L_{\alpha}^+\oplus \bigoplus_{\alpha\in(0,+\infty)}L_\alpha^-$ implies the statement of the lemma.
\end{proof}

\begin{lemma}\label{lemma-3.3}
For any non-negative real number $\alpha$ the sets $L_{\alpha}^+\cup I$ and $L_{\alpha}^-\cup I$ are compact.
\end{lemma}

\begin{proof}
First we show that there exists a non-negative real number $\alpha_0$ such that the sets $L_{\alpha_0}^+\cup I$ and $L_{\alpha_0}^-\cup I$ are compact. We fix an arbitrary open neighbourhood $U(I)$ of the ideal $I$ in $S_2^I$. By Lemma~\ref{lemma-3.2} $L_{\alpha}^+\cup L_{\alpha}^-\subseteq U(I)$ for almost all not finitely many $\alpha\in[0,+\infty)$. Without loss of generality we may assume that the closure $\overline{U(I)}$ of $U(I)$ is a compact subset of $S_2^I$. Fix $\alpha_0\in[0,+\infty)$ such that $L_{\alpha_0}^+\cup L_{\alpha_0}^-\subseteq U(I)$. Since $L_{\alpha_0}^+$ and $L_{\alpha_0}^-$are open subsets of $S_2^I$, we get that
\begin{equation*}
L_{\alpha}^+\cup I=S_2^I\setminus \left(\bigcup_{\alpha_0\neq\alpha\geqslant 0}L_{\alpha}^+ \cup \bigcup_{\alpha>0}L_{\alpha}^-\right) \qquad \mbox{and} \qquad
L_{\alpha}^-\cup I=S_2^I\setminus \left(\bigcup_{\alpha>0}L_{\alpha}^+ \cup \bigcup_{\alpha_0\neq\alpha\geqslant 0}L_{\alpha}^-\right)
\end{equation*}
are closed subsets of $\overline{U(I)}$, and hence they are compact.

We observe that
\begin{equation*}
  (x,x+\alpha_0)\cdot (\alpha_0,\alpha)=(x,x+\alpha) \qquad \mbox{and} \qquad (\alpha,\alpha_0)\cdot(x+\alpha_0,x)=(x+\alpha,x)
\end{equation*}
in $\boldsymbol{B}_{[0,\infty)}$ for any non-negative real numbers $\alpha$, $\alpha_0$ and $x$. This implies that $\rho_{(\alpha_0,\alpha)}(L_{\alpha_0}^+)=L_{\alpha}^+$ and $\lambda_{(\alpha,\alpha_0)}(L_{\alpha_0}^-)=L_{\alpha}^-$, where $\rho_{(\alpha_0,\alpha)}\colon S_2^I\to S_2^I$ and $\lambda_{(\alpha,\alpha_0)}\colon S_2^I\to S_2^I$ are right and left shifts on elements $(\alpha_0,\alpha)$ and $(\alpha,\alpha_0)$, respectively. Since $S_2^I$ is a semitopological semigroup, the sets $\rho_{(\alpha_0,\alpha)}(L_{\alpha_0}^+\cup I)\cup I=L_{\alpha}^+\cup I$ and  $\lambda_{(\alpha,\alpha_0)}(L_{\alpha_0}^-\cup I)\cup I=L_{\alpha}^-\cup I$ are compact.
\end{proof}

\begin{lemma}\label{lemma-3.4}
Let $U(I)$ be any open neighbourhood of the ideal $I$ in $S_2^I$ with the compact closure $\overline{U(I)}$. Then for any non-negative real number $\alpha$ the sets $L_{\alpha}^+\setminus U(I)$ and $L_{\alpha}^-\setminus U(I)$ are compact.
\end{lemma}

\begin{proof}
By Lemma~\ref{lemma-3.3} for any non-negative real number $\alpha$ the sets $L_{\alpha}^+\cup I$ and $L_{\alpha}^-\cup I$ are compact. Since $L_{\alpha}^+\setminus U(I)$ and $L_{\alpha}^-\setminus U(I)$ are closed subsets of $L_{\alpha}^+\cup I$ and $L_{\alpha}^-\cup I$, they are compact.
\end{proof}

Lemmas~\ref{lemma-3.1}, \ref{lemma-3.2}, \ref{lemma-3.3}, and \ref{lemma-3.4} imply the following theorem.

\begin{theorem}\label{theorem-3.5}
Let $S_2^I$ be a Hausdorff locally compact semitopological  semigroup $\boldsymbol{B}^2_{[0,\infty)}$ with an adjoined compact ideal $I$. Then either $I$ is an open subset of $S_2^I$ or the semigroup $S_2^I$ is compact.
\end{theorem}

Next we need some notions for the further construction. For the natural partial order $\preccurlyeq$ on the semigroup $\boldsymbol{B}_{[0,\infty)}$ and any $(a,b)\in\boldsymbol{B}_{[0,\infty)}$ we denote
\begin{align*}
  {\uparrow_{\preccurlyeq}}(a,b)         &=\left\{(x,y)\in\boldsymbol{B}_{[0,\infty)}\colon (a,b)\preccurlyeq(x,y)\right\}; \\
  {\downarrow_{\preccurlyeq}}(a,b)       &=\left\{(x,y)\in\boldsymbol{B}_{[0,\infty)}\colon (x,y)\preccurlyeq(a,b)\right\}; \\
  {\downarrow_{\preccurlyeq}^\circ}(a,b) &={\downarrow_{\preccurlyeq}}(a,b)\setminus\left\{(a,b)\right\}.
\end{align*}

The following statement describes the natural partial order $\preccurlyeq$ on the semigroup $\boldsymbol{B}_{[0,\infty)}$ and it follows from Lemma~1 of \cite{Gutik-Maksymyk=2016}.

\begin{lemma}\label{lemma-3.6}
Let $(a,b)$ and $(c,d)$ be arbitrary elements of the semigroup $\boldsymbol{B}_{[0,\infty)}$. Then the following statements are equivalent:
\begin{enumerate}
  \item[$(i)$] $(a,b)\preccurlyeq(c,d)$;
  \item[$(ii)$] $a\geqslant c$ and $a-b=c-d$;
  \item[$(iii)$] $b\geqslant d$ and $a-b=c-d$.
\end{enumerate}
\end{lemma}

Lemma~\ref{lemma-3.6} implies that for any non-negative real number $\alpha$ the set $L_{\alpha}^+$ coincides with all elements of $\boldsymbol{B}_{[0,\infty)}$ which are comparable with $(0,\alpha)$, and the set $L_{\alpha}^-$ coincides with all elements of $\boldsymbol{B}_{[0,\infty)}$ which are comparable with $(\alpha,0)$ with the respact to the natural partial order $\preccurlyeq$ on the semigroup $\boldsymbol{B}_{[0,\infty)}$. Hence we have that $L_{\alpha}^+={\downarrow_{\preccurlyeq}}(0,\alpha)$ and $L_{\alpha}^-={\downarrow_{\preccurlyeq}}(\alpha,0)$.

Simple calculations and routine verifications show the following proposition.

\begin{proposition}\label{proposition-3.7}
Let $\alpha$ and $\beta$ be non-negative real numbers. Then the following statements hold:
\begin{enumerate}
  \item[$(i)$]   $L_{\alpha}^+\cdot L_{\beta}^+=L_{\alpha+\beta}^+$;
  \item[$(ii)$]  $L_{\alpha}^-\cdot L_{\beta}^-=L_{\alpha+\beta}^-$;
  \item[$(iii)$] $L_{\alpha}^+\cdot L_{\beta}^-=
     \left\{
       \begin{array}{ll}
         L_{\alpha-\beta}^+, & \hbox{if~} \alpha\geqslant\beta; \\
         L_{\beta-\alpha}^-, & \hbox{if~} \alpha\leqslant\beta;
       \end{array}
     \right.
     $
  \item[$(iv)$] $L_{\beta}^-\cdot L_{\alpha}^+={\downarrow_{\preccurlyeq}}(\beta,\alpha)\subseteq
    \left\{
       \begin{array}{ll}
         L_{\alpha-\beta}^+, & \hbox{if~} \alpha\geqslant\beta; \\
         L_{\beta-\alpha}^-, & \hbox{if~} \alpha\leqslant\beta.
       \end{array}
     \right.
   $
\end{enumerate}
\end{proposition}

\begin{lemma}\label{lemma-3.8}
For arbitrary $(a_0,b_0),(a_1,b_1)\in \boldsymbol{B}_{[0,\infty)}$ there exists $(c,d)\in \boldsymbol{B}_{[0,\infty)}$ such that $(a_0,b_0)\cdot(c,d)\preccurlyeq(a_1,b_1)$ $\left[(c,d)\cdot(a_0,b_0)\preccurlyeq(a_1,b_1)\right]$. Moreover, $(a_0,b_0)\cdot(x,y)\preccurlyeq(a_1,b_1)$ $\left[(x,y)\cdot(a_0,b_0)\preccurlyeq(a_1,b_1)\right]$ for any $(x,y)\preccurlyeq(c,d)$ in $\boldsymbol{B}_{[0,\infty)}$.
\end{lemma}

\begin{proof}
We assume that $c\geqslant a_1+a_0+b_0$  and $d=a_0+c-b_0-a_1+b_1$. The semigroup operation of $\boldsymbol{B}_{[0,\infty)}$ implies that
\begin{equation*}
  (a_0,b_0)\cdot(c,d)=(a_0,b_0)\cdot(c,a_0+c-b_0-a_1+b_1)=(a_0-b_0+c,a_0+c-b_0-a_1+b_1).
\end{equation*}
Then $a_0-b_0+c\geqslant a_1$ and
\begin{equation*}
  (a_0-b_0+c)-(a_0+c-b_0-a_1+b_1)=a_0-b_0+c-a_0-c+b_0+a_1-b_1=a_1-b_1,
\end{equation*}
and hence by Lemma~\ref{lemma-3.6} we get that $(a_0,b_0)\cdot(c,d)\preccurlyeq(a_1,b_1)$. The last statement of the lemma follows from Proposition 1.4.7 of \cite{Lawson=1998}. The proof of the dual statement is similar.
\end{proof}

Lemma~\ref{lemma-3.8} implies the following proposition.

\begin{proposition}\label{proposition-3.9}
If $(a_0,b_0)\cdot{\downarrow_{\preccurlyeq}}(c_0,d_0)\subseteq {\downarrow_{\preccurlyeq}}(a_1,b_1)$ $\left[{\downarrow_{\preccurlyeq}}(c_0,d_0)\cdot(a_0,b_0)\subseteq {\downarrow_{\preccurlyeq}}(a_1,b_1)\right]$ for some $(a_0,b_0)$, $(a_1,b_1),(c_0,d_0)\in \boldsymbol{B}_{[0,\infty)}$, then $(a_0,b_0)\cdot{\downarrow_{\preccurlyeq}^\circ}(c_0,d_0)\subseteq {\downarrow_{\preccurlyeq}^\circ}(a_1,b_1)$ $\left[{\downarrow_{\preccurlyeq}^\circ}(c_0,d_0)\cdot(a_0,b_0)\subseteq {\downarrow_{\preccurlyeq}^\circ}(a_1,b_1)\right]$.
\end{proposition}

\begin{example}\label{example-3.10}
Let $S^{\boldsymbol{0}}_2$ be the semigroup $\boldsymbol{B}^2_{[0,\infty)}$ with an adjoined zero $\boldsymbol{0}$.
We extend the topology of  $\boldsymbol{B}^2_{[0,\infty)}$ up to a compact topology $\tau^2_{\textsf{Ac}}$ on the semigroup $S^{\boldsymbol{0}}_2$ in the following way. For any $(a_1,b_1),\ldots,(a_k,b_k)\in \boldsymbol{B}^1_{[0,\infty)}$ we put
\begin{equation*}
  U_{\boldsymbol{0}}[(a_1,b_1),\ldots,(a_k,b_k)]= S^{\boldsymbol{0}}_2\setminus\left({\uparrow_{\preccurlyeq}}(a_1,b_1)\cup\cdots\cup{\uparrow_{\preccurlyeq}}(a_k,b_k)\right)
\end{equation*}
and define
\begin{equation*}
  \mathscr{B}^2_{\textsf{Ac}}(\boldsymbol{0})=\left\{U_{\boldsymbol{0}}[(a_1,b_1),\ldots,(a_k,b_k)] \colon (a_1,b_1),\ldots,(a_k,b_k)\in \boldsymbol{B}_{[0,\infty)}, k\in\mathbb{N}\right\}
\end{equation*}
is the system of open neighbourhoods of zero in $\tau^2_{\textsf{Ac}}$.
\end{example}

\begin{proposition}\label{proposition-3.11}
$(S^{\boldsymbol{0}}_2,\tau^2_{\textsf{Ac}})$ is a compact Hausdorff semitopological semigroup with continuous inversion.
\end{proposition}

\begin{proof}
It is obvious that $\tau^2_{\textsf{Ac}}$ is a compact Hausdorff topology on $S^{\boldsymbol{0}}_2$. Moreover $(S^{\boldsymbol{0}}_2,\tau^2_{\textsf{Ac}})$ is  the one-point Alexandroff compactification of the locally compact space $\boldsymbol{B}^2_{[0,\infty)}$ such that the singleton set $\{\boldsymbol{0}\}$ which consists of the zero of $S^{\boldsymbol{0}}_2$ is its remainder.

By \cite{Ahre=1989}, $\boldsymbol{B}^2_{[0,\infty)}$ is a topological inverse semigroup, and hence  it sufficient to show that the the semigroup operation on  $(S^{\boldsymbol{0}}_2,\tau^2_{\textsf{Ac}})$ is separately continuous at zero.

Fix an arbitrary $U_{\boldsymbol{0}}[(a_1,b_1),\ldots,(a_k,b_k)]\in \mathscr{B}^2_{\textsf{Ac}}(\boldsymbol{0})$.

It is obvious that
\begin{equation*}
\boldsymbol{0}\cdot U_{\boldsymbol{0}}[(a_1,b_1),\ldots,(a_k,b_k)]=U_{\boldsymbol{0}}[(a_1,b_1),\ldots,(a_k,b_k)]\cdot \boldsymbol{0}=\{\boldsymbol{0}\}\subseteq U_{\boldsymbol{0}}[(a_1,b_1),\ldots,(a_k,b_k)].
\end{equation*}

By Lemma~\ref{lemma-3.8} for an arbitrary $(a,b)\in \boldsymbol{B}_{[0,\infty)}$ there exist $(c_1,d_1),\ldots,(c_k,d_k),(x_1,y_1),\ldots,(x_k,y_k)\in \boldsymbol{B}_{[0,\infty)}$ such that $(a,b)\cdot(c_i,d_i)\preccurlyeq(a_i,b_i)$ and  $(x_i,y_i)\cdot(a,b)\preccurlyeq(a_i,b_i)$ for al $i=1,\ldots,k$. By Proposition~\ref{proposition-3.9} we have that $(a,b)\cdot{\downarrow_{\preccurlyeq}^\circ}(c_i,d_i)\subseteq{\downarrow_{\preccurlyeq}^\circ}(a_i,b_i)$ and  ${\downarrow_{\preccurlyeq}^\circ}(x_i,y_i)\cdot(a,b)\subseteq{\downarrow_{\preccurlyeq}^\circ}(a_i,b_i)$ for al $i=1,\ldots,k$. This and Proposition~\ref{proposition-3.7} imply that
\begin{equation*}
(a,b)\cdot U_{\boldsymbol{0}}[(c_1,d_1),\ldots,(c_k,d_k)]\subseteq U_{\boldsymbol{0}}[(a_1,b_1),\ldots,(a_k,b_k)]
\end{equation*}
and
\begin{equation*}
  U_{\boldsymbol{0}}[(x_1,y_1),\ldots,(x_k,y_k)]\cdot(a,b)\subseteq U_{\boldsymbol{0}}[(a_1,b_1),\ldots,(a_k,b_k)],
\end{equation*}
and hence the semigroup operation on $(S^{\boldsymbol{0}}_2,\tau^2_{\textsf{Ac}})$ is separately continuous.

Since $\left(U_{\boldsymbol{0}}[(a_1,b_1),\ldots,(a_k,b_k)]\right)^{-1}=U_{\boldsymbol{0}}[(b_1,a_1),\ldots,(b_k,a_k)$ for any $(a_1,b_1),\ldots,(a_k,b_k)\in \boldsymbol{B}_{[0,\infty)}$ the inversion on $(S^{\boldsymbol{0}}_2,\tau^2_{\textsf{Ac}})$ is continuous.
\end{proof}

Theorem \ref{theorem-3.5} and Proposition \ref{proposition-3.11} imply the following theorem.

\begin{theorem}\label{theorem-3.12}
Let $S^{\boldsymbol{0}}_2$ be a Hausdorff locally compact semitopological  semigroup $\boldsymbol{B}^2_{[0,\infty)}$ with an adjoined zero $\boldsymbol{0}$. Then either $\boldsymbol{0}$ is an isolated point of $S^{\boldsymbol{0}}_2$ or the topology of $S^{\boldsymbol{0}}_2$ coincides with $\tau^2_{\textsf{Ac}}$.
\end{theorem}

Since the bicyclic monoid does not embeds into any Hausdorff compact topological semigroup \cite{Anderson-Hunter-Koch=1965} and the semigroup contains many isomorphic copies of the bicyclic semigroup, Theorems \ref{theorem-3.5} and \ref{theorem-3.12} imply the following corollaries.

\begin{corollary}\label{corollary-3.13}
Let $S_2^I$ be a Hausdorff locally compact topological  semigroup $\boldsymbol{B}^2_{[0,\infty)}$ with an adjoined compact ideal $I$. Then $I$ is an open subset of $S_2^I$.
\end{corollary}

\begin{corollary}\label{corollary-3.14}
Let $S^{\boldsymbol{0}}_2$ be a Hausdorff locally compact topological  semigroup $\boldsymbol{B}^2_{[0,\infty)}$ with an adjoined zero $\boldsymbol{0}$. Then $\boldsymbol{0}$ is an isolated point of $S^{\boldsymbol{0}}_2$.
\end{corollary}

\section{A locally compact semigroup $\boldsymbol{B}^\mathfrak{d}_{[0,\infty)}$ with an adjoined compact ideal}\label{section-4}

Later in this section by $S_\mathfrak{d}^0$ we denote a Hausdorff locally compact semitopological  semigroup which is the semigroup $\boldsymbol{B}^\mathfrak{d}_{[0,\infty)}$ with an adjoined zero $\boldsymbol{0}$.

\begin{lemma}\label{lemma-4.1}
Let $U(\boldsymbol{0})$ be an open neighbourhood of zero with the compact closure $\overline{U(\boldsymbol{0})}$ in $S_\mathfrak{d}^{\boldsymbol{0}}$. Then for any $(a,b)\in\boldsymbol{B}_{[0,\infty)}$ the set ${\uparrow_{\preccurlyeq}}(a,b)\cap U(\boldsymbol{0})$ is finite.
\end{lemma}

\begin{proof}
Suppose to the contrary that there exists an open neighbourhood of zero with the compact closure $\overline{U(\boldsymbol{0})}$ in $S_\mathfrak{d}^{\boldsymbol{0}}$ such that the set ${\uparrow_{\preccurlyeq}}(a,b)\cap U(\boldsymbol{0})$ is infinite. By Remark~\ref{remark-1.1} we have that
\begin{equation*}
  {\uparrow_{\preccurlyeq}}(a,b)=\left\{(x,y)\in \boldsymbol{B}_{[0,\infty)}\colon (a,a)(x,y)=(a,b)\right\},
\end{equation*}
and hence the Hausdorffness and separate continuity of the semigroup operation on $S_\mathfrak{d}^{\boldsymbol{0}}$ imply that ${\uparrow_{\preccurlyeq}}(a,b)$ is a closed subset of $S_\mathfrak{d}^{\boldsymbol{0}}$. Hence, ${\uparrow_{\preccurlyeq}}(a,b)\cap U(\boldsymbol{0})$ is a compact infinite discrete space, a contradiction. The obtained contradiction implies the statement of lemma.
\end{proof}

We observe that since $\boldsymbol{B}^\mathfrak{d}_{[0,\infty)}$ is a discrete subspace of $S_\mathfrak{d}^{\boldsymbol{0}}$, any open neighbourhood of zero $U(\boldsymbol{0})$ is closed.    Lemma~\ref{lemma-4.1} implies the following corollary.

\begin{corollary}\label{corollary-4.2}
For any open compact neighbourhood $U(\boldsymbol{0})$ of zero in $S_\mathfrak{d}^{\boldsymbol{0}}$ and any real number $\alpha\in[0,\infty)$ the set $L_\alpha^+\cap U(\boldsymbol{0})$  $(L_\alpha^-\cap U(\boldsymbol{0}))$ either contains a maximal elements (with the respect to the natural partial order on $\boldsymbol{B}_{[0,\infty)}$) or is empty.
\end{corollary}

\begin{lemma}\label{lemma-4.3}
If $S_\mathfrak{d}^{\boldsymbol{0}}$ admits  the structure of a Hausdorff locally compact semitopological  semigroup with a nonisolated zero, then there exists no an open compact neighbourhood $U(\boldsymbol{0})$ of zero in $S_\mathfrak{d}^{\boldsymbol{0}}$ such that the sets $L_\alpha^+\cap U(\boldsymbol{0})$ and $L_\alpha^+\cap U(\boldsymbol{0})$ are finite for all $\alpha\in[0,\infty)$.
\end{lemma}

\begin{proof}
Suppose to the contrary that there exists an open compact neighbourhood $U(\boldsymbol{0})$ of zero in $S_\mathfrak{d}^{\boldsymbol{0}}$ such that the sets $L_\alpha^+\cap U(\boldsymbol{0})$ and $L_\alpha^+\cap U(\boldsymbol{0})$ are finite for all $\alpha\in[0,\infty)$. Separate continuity of the semigroup operation in $S_\mathfrak{d}^{\boldsymbol{0}}$ implies that there exists an open compact neighbourhood $V(\boldsymbol{0})\subseteq U(\boldsymbol{0})$ of zero in $S_\mathfrak{d}^{\boldsymbol{0}}$ such that
\begin{equation*}
  (1,0)\cdot V(\boldsymbol{0})\cdot (0,1)\subseteq U(\boldsymbol{0}).
\end{equation*}
This inclusion implies that $U(\boldsymbol{0})\setminus V(\boldsymbol{0})$ is an infinite subsets of isolated points, which contradicts the compactness of $U(\boldsymbol{0})$. The obtained contradiction implies the statement of lemma.
\end{proof}

\begin{lemma}\label{lemma-4.4}
If $S_\mathfrak{d}^{\boldsymbol{0}}$ admits  the structure of a Hausdorff locally compact semitopological  semigroup with a nonisolated zero, then for any open compact neighbourhood $U(\boldsymbol{0})$ of zero in $S_\mathfrak{d}^{\boldsymbol{0}}$ the sets $L_\alpha^+\cap U(\boldsymbol{0})$ and $L_\alpha^-\cap U(\boldsymbol{0})$ are infinite for all $\alpha\in[0,\infty)$.
\end{lemma}

\begin{proof}
By Lemma~\ref{lemma-4.3} there exists $\alpha_0\in[0,\infty)$ such that at least one of the sets $L_{\alpha_0}^+\cap U(\boldsymbol{0})$ or $L_{\alpha_0}^-\cap U(\boldsymbol{0})$ is infinite. Without loss of generality we may assume that the set $L_{\alpha_0}^+\cap U(\boldsymbol{0})$ is infinite. Separate continuity of the semigroup operation of $S_\mathfrak{d}^{\boldsymbol{0}}$ implies that there exists an open compact neighbourhood $V(\boldsymbol{0})\subseteq U(\boldsymbol{0})$ of zero in $S_\mathfrak{d}^{\boldsymbol{0}}$ such that $V(\boldsymbol{0})\cdot (\alpha_0,0)\subseteq U(\boldsymbol{0})$. Since $\boldsymbol{B}^\mathfrak{d}_{[0,\infty)}$ is a discrete subspace of $S_\mathfrak{d}^{\boldsymbol{0}}$ and $U(\boldsymbol{0})$ is compact, the set $L_0^+\cap U(\boldsymbol{0})$ is infinite. By the similar way we get that for any $\beta_0\in(0,\infty)$ there exists an open compact neighbourhood $W(\boldsymbol{0})\subseteq U(\boldsymbol{0})$ such that $(\beta_0,0)\cdot W(\boldsymbol{0})\subseteq U(\boldsymbol{0})$ and $W(\boldsymbol{0})\cdot (0,\beta_0)\subseteq U(\boldsymbol{0})$. Since $W(\boldsymbol{0})$ and $U(\boldsymbol{0})$ are compact, $L_0^+\cap W(\boldsymbol{0})$ is an infinite set, and hence the sets $L_{\beta_0}^+\cap U(\boldsymbol{0})$ and $L_{\beta_0}^-\cap U(\boldsymbol{0})$ are infinite
\end{proof}

\begin{lemma}\label{lemma-4.5}
If $S_\mathfrak{d}^{\boldsymbol{0}}$ admits  the structure of a Hausdorff locally compact semitopological  semigroup with a nonisolated zero, then there exists an open compact neighbourhood $U(\boldsymbol{0})$ of zero in $S_\mathfrak{d}^{\boldsymbol{0}}$ such that $L_0^+\cap U(\boldsymbol{0})=\varnothing$.
\end{lemma}

\begin{proof}
By Lemma~\ref{lemma-4.4} for any compact open neighbourhood $U(0)$ of zero in $S_\mathfrak{d}^{\boldsymbol{0}}$ the set $L_0^+\cap U(\boldsymbol{0})$ is infinite. For any positive integer $n_0$ by Lemma~\ref{lemma-4.1} the set ${\uparrow_{\preccurlyeq}}(n_0,n_0)\cap U(\boldsymbol{0})$ is finite. This implies that the set $L_0^+\cap U(\boldsymbol{0})$ is countable. Let $L_0^+\cap U(\boldsymbol{0})=\left\{(a_i,a_i)\colon a_i\in \boldsymbol{B}^\mathfrak{d}_{[0,\infty)}, i\in\omega \right\}$. Put $M=\left\{a_j-a_i\colon i,j\in\omega, i<j\right\}$. The set $M$ is countable as a countable union of a family of countable sets. Then there exists $\alpha\in(0,\infty)\setminus M$. Then for any open compact neighbourhood $V(\boldsymbol{0})\subseteq U(0)$ of zero in $S_\mathfrak{d}^{\boldsymbol{0}}$ the following inclusion $(\alpha,0)\cdot V(\boldsymbol{0})\cdot (0,\alpha)\subseteq U(0)$ does not hold, because $(\alpha,0)\cdot L_0^+\cdot (0,\alpha)\subseteq L_0^+$. This contradicts the separate continuity of the semigroup operation of $S_\mathfrak{d}^{\boldsymbol{0}}$. The obtained contradiction implies the statement of the lemma.
\end{proof}

If we assume that $S_\mathfrak{d}^{\boldsymbol{0}}$ admits  the structure of a Hausdorff locally compact semitopological  semigroup with a nonisolated zero, then we get Lemma~\ref{lemma-4.5} and Lemma~\ref{lemma-4.4}. But the statement of Lemma~\ref{lemma-4.5} contradicts to Lemma~\ref{lemma-4.4}. Hence the following theorem holds.

\begin{theorem}\label{theorem-4.6}
Let $S_{\mathfrak{d}}^0$ be a Hausdorff locally compact semitopological  semigroup which is the semigroup $\boldsymbol{B}^{\mathfrak{d}}_{[0,\infty)}$ with an adjoined zero $\boldsymbol{0}$. Then $\boldsymbol{0}$ is an isolated point of $S_{\mathfrak{d}}^0$.
\end{theorem}

Later we need the following trivial lemma, which follows from separate continuity of the semigroup operation in semitopological semigroups.

\begin{lemma}\label{lemma-4.7}
Let $S$ be a Hausdorff semitopological semigroup and $I$ be a compact ideal in $S$. Then the Rees-quotient semigroup $S/I$ with the quotient topology is a Hausdorff semitopological semigroup.
\end{lemma}

\begin{theorem}\label{theorem-4.8}
Let $S_{\mathfrak{d}}^I=\boldsymbol{B}^{\mathfrak{d}}_{[0,\infty)}\sqcup I$ be a Hausdorff locally compact semitopological  semigroup which is the semigroup $\boldsymbol{B}^{\mathfrak{d}}_{[0,\infty)}$ with an adjoined compact ideal $I$. Then $I$ is an open subset of $S_\mathfrak{d}^I$.
\end{theorem}

\begin{proof}
Suppose to the contrary that $I$ is not open $S_{\mathfrak{d}}^I$. By Lemma \ref{lemma-4.7} the Rees-quotient semigroup $S_{\mathfrak{d}}^I/I$ with the quotient topology $\tau_q$ is a semitopological semigroup. Let $\pi\colon S_{\mathfrak{d}}^I\to S_{\mathfrak{d}}^I/I$ be the natural homomorphism which is a quotient map.
It is obvious that the Rees-quotient semigroup $S_{\mathfrak{d}}^I/I$ is isomorphic to the semigroup $S_{\mathfrak{d}}^{\boldsymbol{0}}$, and hence without loss of generality we may assume that $\pi(S_{\mathfrak{d}}^I) = S_{\mathfrak{d}}^{\boldsymbol{0}}$ and the image $\pi(I)$ is zero of $S_{\mathfrak{d}}^{\boldsymbol{0}}$.

By Lemma 3.16 of \cite{Gutik-Khylynskyi=2022} there exists an open neighbourhood $U(I)$ of the ideal $I$ with the compact closure $\overline{U(I)}$. Since  every point of $\boldsymbol{B}^{\mathfrak{d}}_{[0,\infty)}$ is isolated in $S_{\mathfrak{d}}^I$ we have that $U(I)=\overline{U(I)}$ and its image $\pi(U(I))$ is a
compact-and-open neighbourhood of zero in $S_{\mathfrak{d}}^{\boldsymbol{0}}$. Hence $S_{\mathfrak{d}}^{\boldsymbol{0}}$ is Hausdorff locally compact space. By Theorem~\ref{theorem-4.6}, $\boldsymbol{0}$ is an isolated point of $S_{\mathfrak{d}}^0$. Since $\pi\colon S_{\mathfrak{d}}^I\to S_{\mathfrak{d}}^I/I$  is a quotient map, $I$ is an open subset of $S_\mathfrak{d}^I$.
\end{proof}

\end{document}